\documentclass[a4paper,12pt]{amsart}
\usepackage{geometry}
\usepackage{euscript,eufrak,verbatim}
\usepackage[psamsfonts]{amssymb}
\usepackage[usenames]{color}
\usepackage[usenames]{color}
\usepackage[colorlinks,linkcolor=red,anchorcolor=blue,citecolor=blue]{hyperref}
\newtheorem{theorem}{Theorem}[section]
\newtheorem{lemma}[theorem]{Lemma}

\newtheorem{corollary}[theorem]{Corollary}

\def\PK{\mathbb{P}^{1}(K)}

\def\N{\mathbb{N}}
\def\Z{\mathbb{Z}}
\def\Q{\mathbb{Q}}
\def\C{\mathbb{C}}

\def\PQp{\mathbb{P}^{1}(\Q_p)}
\def\PCp{\mathbb{P}^{1}(\C_p)}

\def\Qp{\mathbb{Q}_{p}}
\def\F{\mathbb{F}}

\begin{document}

\title[ $p$-adic rational maps having empty Fatou set ]{$p$-adic rational maps having empty Fatou set }

\date{} %\today}

\author[A. H. FAN]{Aihua Fan}
\address[A. H. FAN]{School of Mathematics and Statistics, Central China Normal University, 430079, Wuhan, China \& LAMFA, UMR 7352 CNRS, University of Picardie,
33 rue Saint Leu, 80039 Amiens, France}
\email{ai-hua.fan@u-picardie.fr}

\author[S. L. FAN]{Shilei Fan}

\address[S. L. FAN]{School of Mathematics and Statistics \& Hubei Key Laboratory of Mathematical Sciences, Central  China Normal University, Wuhan, 430079, China }
\email{slfan@mail.ccnu.edu.cn}

\author[S. L. FAN]{Yahia Mwanis}
\address[]{School of Mathematics and Statistics, Central  China Normal University, Wuhan, 430079, China}
\email{yahiakfu2015@gmail.com}

\author[Y. F. Wang]{Yuefei Wang}
\address[Y. F. Wang]{College of Mathematics and Statistics, Shenzhen University, Shenzhen 518060, China \& Academy of Mathematics and System Sciences, Chinese Academy of Sciences, Beijing 100190, China }
\email{wangyf@math.ac.cn}

\thanks{A.H. FAN  was supported by NSF of China (Grant No. 12231013 and 11971192); S. L. FAN was partially supported by NSFC (grants No. 12331004  and No. 11971190) and Fok Ying-Tong Education Foundation, China (grant No.171001).  Y. F. WANG was partially supported by NSFC (grants No. 12231013) and NCAMS}

\begin{abstract}
On any  finite algebraic extension $K$ of  the field $\Q_p$ of $p$-adic numbers, there exist rational maps  $\phi\in K(z)$ such that dynamical system $(\mathbb{P}^{1}(K),\phi)$  has empty Fatou set, i.e. the iteration family $\{\phi^n: n\geq 0\}$ is nowhere equicontinuous.
% Any bounded tile of  the field $\Q_p$ of $p$-adic numbers is a  compact open set up to a set of  zero Haar measure. In this note, we present two simple and direct proofs of this fact.
\end{abstract}
\maketitle

\section{Introduction}
Let $p$ be a prime number and  $\Q_p$ be the field of $p$-adic numbers.  Let  $\phi \in \Q_p(z)$ be a rational map of degree at least two, which   induces  a continuous map on the projective line $\PQp$ over $\Q_p$. Recall that  the space $\PQp$  is equipped with the spherical metric.

Let $\C_p$ be the completion of an algebraic closure of $\Q_p$.  Naturally, $\phi$  induces a continuous map from $\PCp$ to itself, where
$\PCp$ is the projective line over $\C_p$ with the spherical metic.   The Fatou-Julia theory  of rational maps $\phi\in \C_p$ on $\PCp$ have been well developed  by  many authors
\cite{Benedetto-Fatou-Component, Benedetto-wandering-domain-polynomial, Hsia-periodic-points-closure, Rivera-Letelier-Dynamique-rationnelles-corps-locaux2003, Silverman-dynamics-book}. Benedetto \cite{BenedettoHM2001ETDS} showed  that $\phi$ has at least one non-repelling fixed point in $\PCp$, which is a dynamical stable point. Hence, the Fatou set of the system $(\PCp,\phi)$ is  not empty.

 On the other hand, the minimality and minimal decomposition  dynamical system $(\mathbb{P}^1(\Q_p),\phi)$ when $\phi$ have good reduction or potential good reduction were investigate in  \cite{FFLW2014,FFLW2017}. Remark that these dynamical systems have empty Julia set which means that the whole space $\PQp$ is the Fatou set.

 In this note, we construct rational maps  $\phi$ such that the Fatou set of the system   $(\PQp,\phi)$ is empty.

For the polynomial $f(z)=(z^p-z)/p$,
 it is known in \cite{WS98} that the system $(\Z_p,f )$ is topologically conjugated  the one-side full shift system $(\Sigma_{p},\sigma)$. See \cite{FLWZ2007} for a general study of the such systems.  Our idea is to construct  a rational map $\phi$ satisfying the following two properties: \\
 \indent (1) $\phi$ maps the points outside $\Z_p$ into $\Z_p$; \\
  \indent(2) $\phi$ is a small perturbation of $f$ on $\Z_p$ such that $(\Z_p,\phi)$ is topologically conjugated to $(\Z_p,f).$

Actually, on each finite extension $K$ of $\Q_p$, we can construct rational maps $\phi$  such that the Fatou set   of the system $(\PK, \phi)$ is empty.  Let $K$ be a finite extension of $\Q_{p}$ and let $d=[K:\Qp]$. Let $e$ be the ramification index of $K$ over $\Q_p$ (see Section \ref{sec2}). It is known that $e\mid d$ and let $f=d/e$.
\begin{theorem}\label{main-thm}
Let $K$ be a finite algebraic extension of  the field $\Q_p$ of $p$-adic numbers.   For positive integers  $m> n$, let
\[\phi(z)=\frac{z^{p^f}-z}{\pi+z^{p^{mf}}-z^{p^{nf}}}.\]
Then the dynamical system $(\PK,\phi)$  has empty  Fatou set.
\end{theorem}
Hence, it follows immediately that there exist rational maps $\phi\in \Q_p(Z)$ such that the dynamics $(\PQp,\phi)$ have empty Fatou set.
 \begin{corollary}\label{main-thm-Qp}
 For positive integer $m>n\geq 1$, let
 \[\phi(z)=\frac{z^{p}-z}{p+z^{p^{m}}-z^{p^{n}}}.\]
  Then the dynamical system $(\PQp,\phi)$ has empty  Fatou set.
\end{corollary}
%Remark that these rational maps are  not transitive  on $\P$ . It is interesting to investigate the existence of rational map

%Hence, we shall give a general version  of  Theorem \ref{main-thm-Qp} for any finite algebraic extension $K$ and prove it in Section \ref{finite-extension-version}.

\section{Preliminaries}\label{sec2}
In this section,  we recall some notations and some useful lemmas concerning the roots of unity and  monomials.
%\subsection{Finite extensions of the field of $p$-adic numbers}\label{finiteextension}
% Let us recall some  notation and facts concerning  the finite extensions of $\Q_p$.
 Let $K$ be a finite
extension of $\Q_{p}$ and let $d=[K:\Qp]$
denote the dimension of $K$ as a vector space over $\Qp$.   The
extended absolute value on $K$ is still denoted by $|\cdot|_{p}$.
For $x\in K^{*}:=K\setminus \{0\}$, $v_{p}(x):=-\log_{p}(|x|_{p})$
defines the valuation of $x$, with convention $v_{p}(0):=\infty$.
One can show that there exists a unique positive integer $e$ which
is called \emph{ramification index} of $K$ over $\Qp$, such that
$$v_{p}(K^*)=\frac{1}{e}\Z.$$(Sometimes, we write the image  of $K^*$ under $| \cdot |_p$ as $| K^* |_p=p^{{\mathbb{Z}}/{e}}$. )

The extension $K$ over $\Qp$ is said to
be \emph{unramified} if $e=1$, \emph{ramified} if $e>1$ and
\emph{totally ramified} if $e=d$.  An element $\pi\in K$ is called a
\emph{uniformizer} if $v_{p}(\pi)=1/e$. For convenience of notation, we
write $$ v_{\pi}(x):=e\cdot v_{p}(x)$$
 for $x\in K$ so that
$v_\pi(x)\in \mathbb{Z}$.
Let $\mathcal{O}_{K}:=\{x\in
K : |x|_p\leq 1\}$, whose elements are called  \emph{integers} of $K$. Let
$\mathcal{P}_{K}:=\{x\in K: |x|_{p}<1\}$, which is the maximal ideal
of $\mathcal{O}_{K}$. The \emph{residual class field} of $K$ is
$\mathbb{K}:=\mathcal{O}_{K}/\mathcal{P}_{K}$. Then
$\mathbb{K}=\mathbb{F}_{p^f}$,  the finite field of $p^{f}$ elements
where $f=d/e$. Let $C=\{c_{0},c_{1},\cdots,c_{p^f-1}\}$ be a fixed
complete set of representatives of the cosets of $\mathcal{P}_{K}$
in $\mathcal{O}_{K}$. Then every $x\in K$ has a unique $\pi$-adic
expansion of the form
\begin{equation}\label{piadic}
   x=\sum_{i=i_{0}}^{\infty}a_i\pi^{i}
\end{equation}
where $i_{0}\in \Z$ and $a_i\in C$ for all $i\geq i_{0}$.

\begin{lemma}[\cite{AK09}, p.56] The integral ring $\mathcal{O}_{K}$ contains the $(p^f-1)$-th roots of unity. Moreover, for any $x\in \mathcal{O}_{K}\setminus\mathcal{P}_{K}$, the  open ball   $D(x,1)$  contains exact  one  $(p^f-1)$-th root of unity.
\end{lemma}

For any integer $k\geq 0$, denote by $M_{k}(z)=z^{p^{kf}}$.
\begin{lemma}\label{mono-lemma1}
Let  $m>n \geq 1 $ be  two positive  integers. Then for any $x\in \mathcal{O}_K$,
\[|M_{m}(x)-M_{n}(x)|_{p}< |\pi|_p.\]
\end{lemma}
\begin{proof}
First assume that  $x\in \pi\mathcal{O}_{K}$. Then $|M_{m}(x)-M_{n}(x)|_{p}=|x|_{p}^{p^{nf}}<|\pi|_p$.

Secondly, assume  $x\in \mathcal{O}_{K}\setminus \pi\mathcal{O}_K.$ There exists a $p^{f}-1$-th root $\xi$ of unity
such that $x=\xi+y$ for some $y\in \pi \mathcal{O}_{K}$. Observing that  $\xi^{p^{mf}}=\xi^{p^{nf}}=\xi$ , we have
\begin{align*}
M_{m}(x)-M_{n}(x)&=\sum_{i=1}^{p^{mf}}C_{p^{mf}}^{i}y^{i}\xi^{p^{mf}-i}  +\sum_{i=1}^{p^{nf}} C_{p^{nf}}^{i}y^{i}\xi^{p^{nf}-i}.
\end{align*}
Note that $p\mid C_{p^{kf}}^{i}$ for each positive integer $k\geq 1$. Since  $y\in \pi \mathcal{O}_{K},$ we have
\[|M_{m}(x)-M_{n}(x)|_p< |y|_p\leq |\pi|_p.\]
\end{proof}

\begin{lemma}\label{mono-lemma2}
Let  $m>n\geq 1$ be two positive integers. Then for  $x,y \in \mathcal{O}_{K}$ with $|x-y|_p<1$, we have
\[|M_{m}(x)-M_{m}(y)|_p<|M_{n}(x)-M_{n}(y)|_p<|x-y|_p.\]
\end{lemma}
\begin{proof}Since   $M_{k+1}(z)=M_{k}(z)^{p^{f}}$, we can write
\begin{align}\label{eq-mono1}
M_{k+1}(x)-M_{k+1}(y)=\Big(M_{k}(x)-M_{k}(y)\Big)\sum_{i=0}^{p^{f}-1} M_{k}(x)^{i}M_{k}(y)^{p^{f}-i}.
\end{align}
Since $|x-y|_p<1$, it is easy to see  $|M_{k}(x)-M_{k}(y)|_p<1,$
which implies that
\begin{align}\label{eq-mono2}
 \sum_{i=0}^{p^{f}-1} M_{k}(x)^{i}M_{k}(y)^{p^{f}-i}\in \pi \mathcal{O}_{K}.
\end{align}
Hence, \eqref{eq-mono1} and \eqref{eq-mono2} implies
\[|M_{k+1}(x)-M_{k+1}(y)|_p <|M_{k}(x)-M_{k}(y)|_p.\]
 By induction on $m$,  for $m>n$ we have
\[|M_{m}(x)-M_{m}(y)|_p<|M_{n}(x)-M_{n}(y)|_p.\]
\end{proof}

\begin{corollary}\label{mono-cor}
Let $x,y \in \mathcal{O}_{K}$ with $|x-y|_p<1$.
For $k\geq 1$,
\[|M_{k}(x)-M_{k}(y)|_p\leq |\pi|_p^{k}  |x-y|_p.\]
\end{corollary}

%\begin{lemma} Let $u,v\in \pi \mathcal{O}_k$. Then we have
%\[|\frac{1}{1+u}-\frac{1}{1+v}|_p=|u-v|_p.\]
%\end{lemma}
%\begin{proof}

%$\end{proof}

\begin{lemma} \label{frac-lemma}
Let $x,y \in \mathcal{O}_K$ and let $u,v\in \pi \mathcal{O}_k$.
If $|u-v|_p\leq |x-y|_p$, then  we have
\[\Big{|}\frac{x}{1+u}-\frac{y}{1+v}\Big{|}_p\leq |x-y|_p.\]
	Moreover, if $|u-v|_p< |x-y|_p$, the equality holds.
\end{lemma}

\begin{proof}Since $u,v\in \pi \mathcal{O}_k$, we have $|1+u|_p=|1+v|_p=1$. Hence,
\begin{align*}
\Big|\frac{x}{1+u}-\frac{y}{1+v}\Big|_p =  |x(1+v)-y(1+u)|_p=|x(1+u)-y(1+u)+x(v-u)|_p.
\end{align*}
The inequality follows, because $|x(v-u)|_p\leq   |x-y|_p$.

If $|u-v|_p< |x-y|_p$, then $|x(v-u)|_p<  |x-y|_p$. Hence, we get the equality.
%\[\Big|\frac{x}{1+u}-\frac{y}{1+v}\Big|_p= |x-y|_p.\]
\end{proof}

%\begin{lemma}[\cite{AK09}, p.57] Let $n$ be an integer that is relatively prime to $p^f-1$. If $x^n=1$, then $x\equiv 1\ ({\rm mod}\   \pi)$.
%\end{lemma}
%\begin{lemma}[\cite{AK09}, p.57] If $x\equiv 1\ ({\rm mod}\   \pi)$ and  $x^n=1$, then  $n$ is divisible by a
%power of $p$ and $x$ is a root of unity for that power of $p$.
%\end{lemma}
%\begin{theorem}[\cite{Rob}, p.107]
%Let $\zeta$ be a root of unity in the field $K$ having oder $p^t(t\geq 1)$. Then $|\zeta-1|_p=|p|^{1/\varphi(p^t)}<1$, where
%$\varphi(p^t)=p^{t-1}(p-1)$ denotes the Euler $\varphi$-function.
%\end{theorem}

%\begin{lemma}
%Let $f(z)=\sum_{i=1}^{+\infty}a_iz^i, a_i \in \mathcal{O}_{K}$. Then $f$ converges on $ \mathcal{O}_{K}(0,1)$.
%Moreover, if $|a_1|_p=1$, then $f$ is an isometry on $D(0,1)$, i.e.
%\[ |f(x)-f(y)|_p=|x-y|_p, \quad   \forall  x,y \in  D(0,1).\]
%\end{lemma}

%\subsection{xxxx}

\section{Rational  maps with empty Fatou set}\label{finite-extension-version}

Let $K$ be a finite extension of $\Q_p$ and let $\phi\in K(z)$ be a rational map of degree at least $2$.
Then $\phi$ induce a continuous map from the projective line $\PK$ over $K$ to itself.
Here, we equip the projective line $(\PK)$ with the spherical metric $\rho(\cdot,\cdot)$, i.e.
for $P=[x_1,y_1], Q=[x_2,y_2]\in  \PK$,
 \[\rho(P,Q)=\frac{|x_1y_2-x_2y_1|_p}{\max\{|x_1|_{p},|y_1|_{p}\}\max\{|x_2|_{p},|y_{2}|_{p}\}}.\]

Consider the dynamical system $(\PK,\phi)$.
The \textit{Fatou set} $\mathcal{F}(\phi)$ of the system is defined to be the subset of $\PK$ consisting of the points having a neighborhood on which the family of iterates $\{\phi^n\}_{n\ge 0}$ is equicontinuous with respect to the spherical metric. The \textit{Julia set} of the system $(\PK,\phi$ is the complement of $\mathcal{F}(\phi)$ in $\PK$, denoted by $\mathcal{F}(\phi)$.

In this section, we show that there exist rational maps in $K(z)$ with empty Fatou set for  their induced  dynamical systems on $\PK$.

The following lemma shows that $\phi$ is locally scaling on $\mathcal{O}_K$.
\begin{lemma}
For  $x,y \in \mathcal{O}_K$ with $|x-y|_p<1$, we have
\[|\phi(x)-\phi(y)|_p=p^{1/e}|x-y|_p.\]
\end{lemma}
\begin{proof}
Recalling that $M_{k}(z)=z^{p^{kf}}$, we have
\begin{align*}
\phi(x)-\phi(y)= \frac{1}{\pi}\cdot \left(\frac{M_1(x)-x}{1+\frac{M_{m}(x)-M_{n}(x)}{\pi}}- \frac{M_1(y)-y}{1+\frac{M_{m}(y)-M_{n}(y)}{\pi}} \right).
\end{align*}
By Lemma \ref{mono-lemma2}, under the assumption $|x-y|_p<1$,  we have
\begin{align}\label{ineq-mono1}
\Big|\frac{M_{m}(x)-M_{n}(x)}{\pi}-\frac{M_{m}(y)-M_{n}(y)}{\pi}\Big|_p=\Big|\frac{M_{n}(x)-M_{n}(y)}{\pi}\Big|_p
\end{align}
and
\begin{align}\label{ineq-mono2}
|(M_{1}(x)-x)-(M_{1}(y)-y)|_p=|x-y|_p.
\end{align}
Since $n\geq 2$, by Corollary \ref{mono-cor},
\[\Big|\frac{M_{n}(x)-M_{n}(y)}{\pi}\Big|_p<|x-y|_p,\]
which  together with  \eqref{ineq-mono1} and  \eqref{ineq-mono2} implies that
\[\Big|\frac{M_{m}(x)-M_{n}(x)}{\pi}-\frac{M_{m}(y)-M_{n}(y)}{\pi}\Big|_p<|(M_{1}(x)-x)-(M_{1}(y)-y)|_p.\]
By Lemma \ref{mono-lemma1}, we have
\[M_{m}(x)-M_{n}(x), M_{m}(x)-M_{n}(x)\in \pi^2\mathcal{O}_K,\]
which is equivalent to
\[\frac{M_{m}(x)-M_{n}(x)}{\pi}, \frac{M_{m}(x)-M_{n}(x)}{\pi}\in \pi\mathcal{O}_K.\]
Hence, by Lemma \ref{frac-lemma}, we have
\[|\phi(x)-\phi(y)|_p=|1/\pi|_p \cdot |x-y|_p=p^{1/e}|x-y|_p.\]
\end{proof}

\begin{proof}[The proof of Theorem \ref{main-thm}]
Observe  that $m\geq 2$. Then we have
\[\phi(\PK\setminus \mathcal{O}_K)\subset \pi\mathcal{O}_K.\]
% $x\in \PK\setminus \mathcal{O}_{K}$, we have
%\[|\phi(x)|_p=|x|_p^{p^{f}-p^{mf}}<1.\]
The same argument in Lemma \ref{mono-lemma1} implies that
\[\forall x\in \mathcal{O}_{K}, \quad |M_1(x)-x|_p\leq |\pi|_p.\]
Also, by Lemma  \ref{mono-lemma1}, we have
\[\phi(\mathcal{O}_{K})\subset\mathcal{O}_K.\]
%\[\forall x\in\mathcal{O}_{K},\quad |M_{m}(x)-M_{n}(x)|_p< \pi.\]
%Hence,  it follows that
%\[\forall x\in \mathcal{O}_{K},  \quad \phi(x)\in \mathcal{O}_{K}. \]
Then  we shall determine the symbolic dynamics of the dynamical system  $( \mathcal{O}_{K},\phi)$.
Denote the reduction modulo $\pi$ of $x\in \mathcal{O}_{K}$ by $\overline{x} \in \F_{p^{f}}$ and set
\[ \Sigma_{\F_{p^f}}=(\F_{p^f})^{\N}\]
by which we mean the direct product of a countable  copies of $\F_{p^{f}}$, each with the discrete topology.
Now we shall show that $(\mathcal{O}_K,\phi)$ is topologically conjugated to $( \Sigma_{\F_{p^f}},\sigma)$, where $\sigma$ is the  shift
map on $\Sigma_{\F_{p^f}}$, i.e. $\sigma:\Sigma_{\F_{p^f}}\to \Sigma_{\F_{p^f}}$ by $(\sigma (x))_i=x_{i+1}$  for all $x=(x_i)\in \Sigma_{\F_{p^f}}$.
Define a map $H:\mathcal{O}_{K}\to \Sigma_{\F_{p^f}} $ by
\[(H(z))_i= \overline{\phi^i(z)}, \quad  \forall z\in \mathcal{O}_K.\]
By the  same argument in \cite[Theorem 1]{WS98},  one can see that  $H$ is a bijective.
It is clear that $H$ is continuous and $\sigma \circ H= \phi\circ H$.

Hence, we obtain that  the Julia set of  dynamical system $(\PK,\phi)$ is the whole space $\PK$.
\end{proof}

\bibliographystyle{siam}

\begin{thebibliography}{10}

\bibitem{AK09}
{\sc V.~Anashin and A.~Khrennikov}, {\em Applied algebraic dynamics}, vol.~49
  of De Gruyter Expositions in Mathematics, Walter de Gruyter \& Co., Berlin,
  2009.

\bibitem{Benedetto-Fatou-Component}
{\sc R.~L. Benedetto}, {\em Fatou components in p-adic dynamics}, ProQuest LLC,
  Ann Arbor, MI, 1998.
\newblock Thesis (Ph.D.)--Brown University.

\bibitem{BenedettoHM2001ETDS}
\leavevmode\vrule height 2pt depth -1.6pt width 23pt, {\em Hyperbolic maps in
  {$p$}-adic dynamics}, Ergodic Theory Dynam. Systems, 21 (2001), pp.~1--11.

\bibitem{Benedetto-wandering-domain-polynomial}
\leavevmode\vrule height 2pt depth -1.6pt width 23pt, {\em Wandering domains in
  non-{A}rchimedean polynomial dynamics}, Bull. London Math. Soc., 38 (2006),
  pp.~937--950.

\bibitem{FFLW2014}
{\sc A.~Fan, S.~Fan, L.~Liao, and Y.~Wang}, {\em On minimal decomposition of
  {$p$}-adic homographic dynamical systems}, Adv. Math., 257 (2014),
  pp.~92--135.

\bibitem{FFLW2017}
\leavevmode\vrule height 2pt depth -1.6pt width 23pt, {\em Minimality of
  {$p$}-adic rational maps with good reduction}, Discrete Contin. Dyn. Syst.,
  37 (2017), pp.~3161--3182.

\bibitem{FLWZ2007}
{\sc A.~Fan, L.~Liao, Y.~F. Wang, and D.~Zhou}, {\em {$p$}-adic repellers in
  {$\Bbb Q_p$} are subshifts of finite type}, C. R. Math. Acad. Sci. Paris, 344
  (2007), pp.~219--224.

\bibitem{Hsia-periodic-points-closure}
{\sc L.-C. Hsia}, {\em Closure of periodic points over a non-{A}rchimedean
  field}, J. London Math. Soc. (2), 62 (2000), pp.~685--700.

\bibitem{Rivera-Letelier-Dynamique-rationnelles-corps-locaux2003}
{\sc J.~Rivera-Letelier}, {\em Dynamique des fonctions rationnelles sur des
  corps locaux}, Ast\'erisque,  (2003), pp.~xv, 147--230.
\newblock Geometric methods in dynamics. II.

\bibitem{Silverman-dynamics-book}
{\sc J.~H. Silverman}, {\em The arithmetic of dynamical systems}, vol.~241 of
  Graduate Texts in Mathematics, Springer, New York, 2007.

\bibitem{WS98}
{\sc C.~F. Woodcock and N.~P. Smart}, {\em {$p$}-adic chaos and random number
  generation}, Experiment. Math., 7 (1998), pp.~333--342.

\end{thebibliography}

%\bibliographystyle{alpha}
%
%\begin{thebibliography}{99}
%\bibitem{AK09}
% V.~S. Anashin and A. Khrennikov,
%\newblock {\em Applied Algebraic Dynamics,} de  {\em Gruyter Expositions in Mathematics}.
%{49}. Walter de Gruyter \& Co., Berlin, 2009.
%
%
%\bibitem{Mahler81}
%K. Mahler,
%\newblock {\em {$p$}-adic numbers and their functions}, volume~76 of {\em
%  Cambridge Tracts in Mathematics}.
%\newblock Cambridge University Press, Cambridge, second edition, 1981.
%
%\bibitem{Rob}
%A.~Robert.
%\newblock {\em A course in {$p$}-adic analysis}, volume 198 of {\em Graduate
%  Texts in Mathematics}.
%\newblock Springer-Verlag, New York, 2000.
%
%\end{thebibliography}

\end{document}